\documentclass[11pt,leqno]{amsart}

\usepackage{amsthm, amsmath, amssymb}
\usepackage{enumitem, geometry, caption, graphicx, array}
\usepackage{hyperref, url, fontenc, inputenc}

\usepackage{booktabs, cite, float, footmisc, multicol, xcolor}

\usepackage{newtxmath, newtxtext, kvoptions, nag, ragged2e}
\usepackage{pdf14, pdftexcmds, xpatch, zref-base}

\usepackage{xy}
\xyoption{all}

\theoremstyle{remark}

\theoremstyle{plain}

\newcounter{theoremintro} 

\newtheorem{introtheorem}[theoremintro]{Theorem}
\newtheorem{introcorollary}[theoremintro]{Corollary}

\newtheorem*{definition*}{Definition} 
\newtheorem*{theorem*}{Theorem} 
\newtheorem*{lemma*}{Lemma}
\newtheorem*{corollary*}{Corollary} 
\newtheorem*{hypothesis*}{Hypothesis}

\newtheorem{theorem}[subsection]{Theorem} 
\newtheorem{lemma}[subsection]{Lemma}
\newtheorem{corollary}[subsection]{Corollary}
\newtheorem{proposition}[subsection]{Proposition}

\newtheorem{claim}[subsection]{Claim}

\theoremstyle{definition}

\theoremstyle{remark}

\newtheorem{remark}[subsection]{Remark}

\numberwithin{equation}{section}

\newcommand{\R}{\mathbb{R}}
\newcommand{\C}{\mathbb{C}} 

\newcommand{\Z}{\mathbb{Z}}

\newcommand{\oN}{\overline{N}}

\newcommand{\SO}{\mathrm{SO}}
\newcommand{\SU}{\mathrm{SU}}
\newcommand{\Sp}{\mathrm{Sp}}

\newcommand{\F}{\mathbb{F}}

\newcommand{\cN}{\mathcal{N}}

\newcommand{\s}[1]{\langle #1 \rangle}
\newcommand{\hill}{\mathcal{H}}

\begin{document}

\title[slow growth representations of Sp(n, 1) at the edge of Cowling's strip]{slow exponential growth representations of Sp(n, 1) at the edge of Cowling's strip} 
\author{Pierre Julg}
\author{Shintaro Nishikawa}

\address{P.J.: Institut Denis Poisson, Universit\'e d'Orl\'eans, Collegium Sciences et Techniques B\^atiment de math\'ematiques, Rue de Chartres B.P. 6759 F-45067 Orl\'eans Cedex 2 - France} 
\email{pierre.julg@univ-orleans.fr}

\address{S.N.: Mathematisches Institut, Fachbereich Mathematik und Informatik der Universit\"at M\"unster, Einsteinstrasse 62, 48149 M\"unster, Germany} 
\email{snishika@uni-muenster.de}

\thanks{S.N.\ was supported by the Deutsche Forschungsgemeinschaft (DFG, German Research Foundation) under Germany's Excellence Strategy EXC 2044-390685587, Mathematics M\"unster: Dynamics-Geometry-Structure.}


\subjclass[2020]{Primary 22E46}

\keywords{Sp(n, 1), Uniformly bounded representations, Slow exponential growth representations}

\date{\today}

\maketitle

\begin{abstract} We obtain a slow exponential growth estimate for the spherical principal series representation $\rho_s$ of the Lie group $\Sp(n, 1)$ at the edge $(\mathrm{Re}(s)=1)$ of Cowling's strip $(|\mathrm{Re}(s)|<1)$ on the Sobolev space $\hill^\alpha(G/P)$ when $\alpha$ is the critical value $Q/2=2n+1$. As a corollary, we obtain a slow exponential growth estimate for the homotopy $\rho_s$ ($s\in [0, 1]$) of the spherical principal series which is required for the first author's program for proving the Baum--Connes conjecture with coefficients for $\Sp(n,1)$.
\end{abstract}

\tableofcontents

 \section*{Introduction}
 
Let $G=\SO_0(n,1)$, resp. $\SU(n,1)$, resp. $\Sp(n,1)$, resp. $\mathrm{F}_{4(-20)}$, a simple Lie group of real rank one. Consider its spherical principal series representation $\rho_s(g)$ on $C^\infty(G/P)$ for $s\in \C$ where $P$ is a minimal parabolic subgroup of $G$. Let us normalize the parameters $s$ so that $\rho_{it}(g)$ is unitary for $t\in \R$ with respect to the canonical $L^2$-norm on $G/P$ and so that it contains the trivial sub-representation at $s=1$. More specifically, we have
 \[
 \rho_s(g)=\lambda_g^{Q/2(1-s)}\rho(g)
 \]
 on $C^\infty(G/P)$ where $\rho(g)$ is the left-translation by $g$, $\lambda_g$ is a cocycle, and $Q=n-1, 2n, 4n+2, 22$ respectively.

The picture of the unitary dual is rather different between the case of $\SO_0(n,1)$ and $\SU(n,1)$, which have the Haagerup property (or Gromov's a-T-menability), and the case of $\Sp(n,1)$ and $\mathrm{F}_{4(-20)}$ which have Kazhdan's property (T) (see \cite[Chapter 5]{GJV19}). The complementary series (i.e. the $\rho_s$, $s$ real,  which are unitary for a suitable Hilbert scalar product) is defined for $-1<s<1$ in the first case, whereas in the second case a gap appears between the trivial representation and the complementary series, which ranges for $-{2n-1\over 2n+1}<s<{2n-1\over 2n+1}$ for $\Sp(n,1)$, $-{5\over 11}<s<{5\over 11}$ for $\mathrm{F}_{4(-20)}$ (see \cite{Kostant1960}).

However, it was pointed out in the 1980's by Michael Cowling that if one replaces ``unitary'' by ``uniformly bounded'', the difference disappears (see \cite{Cowling2010}). Namely, the representation $\rho_s$ for $s$ in the strip $-1<\mathrm{Re}(s)<1$ becomes a uniformly bounded representation of $G$ on a Hilbert space, if we consider an appropriate Sobolev space. Here, a representation is uniformly bounded in a sense that the operator norm  $||\rho_s(g)||$ of $\rho_s(g)$ is bounded by a constant $C>0$ which is independent of $g$ in $G$.

 Analysis on Heisenberg groups plays an important role as a nilpotent subgroup $N$ of $G$ provides the so-called open picture of the principal series representation $\rho_s$. For $-1<\mathrm{Re}(s)<1$, the representation $\rho_s$ is equivalently represented on the homogeneous Sobolev space $\dot\hill^\alpha(N)$ on $N$ with respect to a sub-Laplacian on $N$. This picture has an advantage that it uses the canonical norm for which $\rho_s(g)$ is unitary for $g$ in $P$. It then suffices to show that a single element $w$ in $G$ is bounded to see $\rho_s$ is uniformly bounded. On the other hand, for $s$ outside the strip $-1<\mathrm{Re}(s)<1$, $w$ is not bounded with respect to such a norm and hence the representation itself would not be well-defined.

In the compact picture, Cowling's result reads as follows: the representation $\rho_s$ on the Sobolev space $\hill^\alpha(G/P)$ for $s$ in the strip $-1<\mathrm{Re}(s)<1$ is uniformly bounded for $\alpha=(Q/2)s$ (see \cite{ACD04}).  Here, the Sobolev space $\hill^\alpha(G/P)$ is defined as the completion of $C^\infty(G/P)$ with respect to the Euclidean norm $||(1+\Delta_E)^{\alpha/2}\xi||_{L^2(G/P)}$ where $\Delta_E$ is a K-invariant sub-Laplacian on $G/P$ for a maximal compact subgroup $K$ of $G$. The compact picture has an advantage that the representation $\rho_s$ on $\hill^\alpha(G/P)$ is well-defined (bounded) for all $s$, in particular for $s=1$.
 
 In the first author's program for proving the Baum--Connes conjecture with coefficients for $\Sp(n,1)$ \cite{Julg19}, it is crucial to consider the growth of the operator norm $||\rho_s(g)||_{\hill^\alpha(G/P)\to \hill^\alpha(G/P)}$ of the representation $\rho_s(g)$ on $\hill^\alpha(G/P)$ for $\alpha=(Q/2)s$ and $0\leq s\leq 1$ as $s$ approaches to $1$ (see \cite[Section 8.3]{Julg19}). In particular, we would like to show that the homotopy $\rho_s(g)$ ($s\in [0, 1]$) of representations has slow exponential growth in a sense that for any $\epsilon>0$, there is $C=C(\epsilon)>0$ such that 
 \[
||\rho_s(g)||_{\hill^{(Q/2)s}(G/P)\to \hill^{(Q/2)s}(G/P)} \leq C e^{\epsilon l(g)}
\]
for all $s\in [0, 1]$ and for all $g$ in $G$. Here, $l(g)$ is a K-bi-invariant length function on $G$ defined as $l(g)=d_{G/K}(gK, K)$, i.e. $l(ka_tk')=|t|$ for any $k, k'$ in $K$ and $t\in \R$. 
 
With an application to this problem in mind, Astengo, Cowling and Di Blasio obtained a similar type of estimates for $\rho_s(g)$ on $\hill^\alpha(G/P)$. They showed that \cite[Theorem 5.1]{ACD04} for any $\alpha \in (-Q/2, Q/2)$ fixed, there is $C>0$ such that
 \[
 ||\rho_s(g)||_{\hill^{\alpha}(G/P)\to \hill^{\alpha}(G/P)}\leq Ce^{((Q/2)\mathrm{Re}(s)-\alpha)l(g)}
 \]
 for all $s\in \C$ and for all $g$ in $G$. 
 
 Although the above estimate suggests that the Sobolev space $\hill^{Q/2}(G/P)$ should provide a Hilbert space for which $\rho_s(g)$ has slow exponential growth when $\mathrm{Re}(s)=1$, the estimate itself is not enough to conclude that the slow exponential growth of the homotopy $\rho_s(g)$ ($s\in [0, 1]$).  Because of this, we shall study the growth of the representation $\rho_s(g)$ on $\hill^\alpha(G/P)$ when $\alpha$ is the critical value $\alpha=Q/2$. The following is our main result:

 \begin{introtheorem}(See Theorem \ref{thm_main}) For any $\epsilon>0$, there is $C=C(\epsilon)>0$ such that for all $s$ satisfying $\mathrm{Re}(s) = 1$, we have the following upper-bound for the operator norm $||\rho_s(g)||_{\hill^{Q/2}(G/P)\to \hill^{Q/2}(G/P)}$ of $\rho_s(g)$ on $\hill^{Q/2}(G/P)$ for all $g$ in $G$:
\[
||\rho_s(g)||_{\hill^{Q/2}(G/P)\to \hill^{Q/2}(G/P)} \leq C e^{\epsilon l(g)} (1+|\mathrm{Im}(s)|)^{Q/2}.
\]
\end{introtheorem}

As a corollary, by a simple application of complex interpolation, we show the desired slow exponential growth estimate for the homotopy $\rho_s(g)$ ($s\in [0, 1]$).

\begin{introcorollary}(See Corollary \ref{cor_main}) For any $\epsilon>0$, there is $C=C(\epsilon)>0$ such that for any $s\in [0, 1]$, we have for all $g$ in $G$,
\[
||\rho_s(g)||_{\hill^{(Q/2)s}(G/P)\to \hill^{(Q/2)s}(G/P)} \leq C e^{\epsilon l(g)}.
\]
\end{introcorollary}

We end our introduction by explaining the current status of the first author's program for proving the Baum--Connes conjecture with coefficients for $\Sp(n,1)$ \cite{Julg19}. In \cite{Julg19}, a BGG-cycle $(H, \pi, F)$ for the Kasparov's ring $R(G)=KK^G(\C, \C)$ was constructed. Let us call its class $\gamma_r=[H, \pi, F]$ in $R(G)$. In \cite{Julg19}, the remaining problems were to prove the following:
\begin{enumerate}
\item The element $\gamma_r$ is equal to the gamma element $\gamma$ in $R(G)$ \cite[Conjecture 1]{Julg19};
\item  The element $\gamma_r$ is equal to the identity $1_G$ in $R_\epsilon(G)$ for any $\epsilon>0$ \cite[Conjecture 2]{Julg19}.
\end{enumerate}
See \cite[Section 1]{Julg19} for the definition of the gamma element $\gamma$. It suffices to say that it is constructed for all almost connected groups and that the construction is based on the de-Rham complex on the Riemannian symmetric space $G/K$, whereas the element $\gamma_r$ is based on the BGG-complex on the spherical variety $G/P$. The ring $R_\epsilon(G)$ in the second item is defined in the same way as $R(G)$ except that the representations of $G$ on Hilbert spaces for defining cycles $(H, \pi, F)$ may have exponential  growth $||\pi(G)||\leq Ce^{\epsilon}$ (to be precise, one has to fix the parameter $C_\epsilon>0$ for each $\epsilon$). As explained in \cite{Julg19}, the validity of the two items imply the Baum--Connes conjecture with coefficients for $\Sp(n, 1)$. 
 
 Our main result (Corollary \ref{cor_main}) implies that the second item indeed holds.
 
\begin{introtheorem} The item (2) holds.
\end{introtheorem}

Therefore, we are left with the problem (1), which is essentially reduced to a problem of showing the compactness of the commutator $[S_0, f]$ for any continuous function $f$ on the disk $G/K\cup G/P$ and for a certain bounded operator $S_0$, which is essentially a Poisson transform (see \cite[Section 7.4]{Julg19}).

Here, it is perhaps worth to recall the following hypothesis made by Kasparov in \cite{Kasparov85}.

\begin{hypothesis*}[{\cite[Section 5.11]{Kasparov85}}] For any almost connected group $G$, the restriction to a maximal compact subgroup $K$ determines an isomorphism $R_r(G)\cong R(K)$.
\end{hypothesis*}

Here, the ring $R_r(G)$ is defined in the in the same way as $R(G)$ except the representations of $G$ have to be weakly contained in the left-regular representations. There is a canonical map $R_r(G)\to R(G)$ which is an isomorphism precisely when $G$ is $K$-amenable and is never surjective for $G$ with Kazhdan's property (T). If we denote the support of the gamma element by $\gamma R(G)$, there is a natural map
\[
\gamma R(G) \to R_r(G).
\]
Since we already know $\gamma R(G)\cong R(K)$, the hypothesis of Kasparov is equivalent to saying this map is an isomorphism. If this is the case, the item (1), $\gamma_r=\gamma$, would follow since both are elements in $R_r(G)$ that have the property that its restriction to $R(K)$ is the identity $1_K$. In other words, the validity of the hypothesis would immediately imply the item (1), and hence the Baum--Connes conjecture with coefficients for $\Sp(n, 1)$. However, the hypothesis would be strictly harder to prove than the item (1), let alone we are not sure whether it is true in general.

\section*{Acknowledgements} 
The first author would like to thank Michael Cowling for many fruitful discussions during the last 30 years. He is also grateful to Nigel Higson and Vincent Lafforgue for pointing out the necessity to consider slow exponential growth representations in the context of $K$-theory of group $C^*$-algebras and Baum-Connes conjecture.

The second author would like to thank Nigel Higson for teaching him about representation theory during the past years. 

Part of this research was begun within the online Research Community on Representation Theory and Noncommutative Geometry sponsored by the American Institute of Mathematics.

 \section{Preliminaries}
Let $\F=\R, \C$ or $\mathbb{H}$ be the field of real numbers, complex numbers or quaternions. For $z\in \F$, we define
\[
|z|^2=z^\ast z, \,\,\, \mathrm{Re}(z)=\frac{z+z^\ast}2, \,\, \mathrm{Im}(z)=\frac{z-z^\ast}2.
\]
We also write $\mathrm{Im}(\F)\subset \F$ to be the image of $\mathrm{Im}(\,\,)$ on $\F$.

A sesquilinear form $q$ on a right vector space $\F^{n+1}$ over $\F$ is given by
   \[
   q(z, w)=-\bar z_0 w_0 + \sum_{j=1}^{j=n}\bar z_j w_j
   \]
   for $z, w$ in $\F^{n+1}$. 
   
Let $O(q)$ be the group of $(n+1)\times (n+1)$ matrices over $\F$ which act on $\F^{n+1}$ from left and preserve the quadratic form $q$.

The Lie group $\SO_0(n, 1)$ is the connected component of the identity of $O(q)$ for $\F=\R$, $\SU(n, 1)$ is $O(q)\cap SL(n+1, \mathbb{C})$ for $\F=\mathbb{C}$ and $\Sp(n, 1)$ is $O(q)$ for $\F=\mathbb{H}$.  

In this paper, we consider $G=\Sp(n, 1)$ ($n\geq2$) and thus $\F=\mathbb{H}$ but all the results have obvious analogues for $\SO_0(n, 1)$ and $\SU(n, 1)$.  The closed subgroup $K$ of $G$ that preserves the canonical Euclidean metric on $\F^{n+1}$ is a maximal compact subgroup of $G$.

The Lie group $G$ naturally acts on the projective space $\mathbb{P}(\F^{n+1})$ over $\F$ and we have
\begin{align*}
G\cdot [1, 0, \cdots, 0]^T & = \{ \, [z_0, z_1, \cdots, z_n]^T \in \mathbb{P}(\F^{n+1})  \,  \mid  \,\,\, \sum_{1\leq j \leq n}|z_j|^2<|z_0|^2 \, \} \\
&  =  \{ \, [1, z_1, \cdots, z_n]^T \in  \mathbb{P}(\F^{n+1})  \mid  \,\,\, \sum_{1\leq j \leq n}|z_j|^2<1 \, \}.
\end{align*}
The isotropy subgroup of $G$ at the point $[1, 0, \cdots, 0]^T$ is $K$. In this way, $G/K$ can be viewed as the disk $\mathbb{D}^{4n}$ in $\F^{n}$.

The boundary of $G/K$ in $\mathbb{P}(\F^{n+1})$ is 
\begin{align*}
G\cdot [1, 0, \cdots, 0, 1]^T & = \{ \, [z_0, z_1, \cdots, z_n]^T \in \mathbb{P}(\F^{n+1})  \,  \mid  \,\,\, \sum_{1\leq j \leq n}|z_j|^2=|z_0|^2 \, \} \\
&  =  \{ \, [1, z_1, \cdots, z_n]^T \in  \mathbb{P}(\F^{n+1})  \mid  \,\,\, \sum_{1\leq j \leq n}|z_j|^2=1 \, \}.
\end{align*}
The isotropy subgroup $P$ of $G$ at the point $[1, 0, \cdots, 0, 1]^T$ is a minimal parabolic subgroup of $G$. In this way, $G/P$ can be viewed as the sphere $S^{4n-1}$ in $\F^{n}$.

Let $A$ be a closed subgroup of $G$ defined as
   \[
   A= \{\, a_t\in G \mid t\in \R \, \}, \,\,\,    a_t = \begin{bmatrix} \cosh t & 0 & \sinh t \\ 0 & 1 & 0 \\ \sinh t & 0 & \cosh t \end{bmatrix}=U \begin{bmatrix} e^{-t} & 0 & 0 \\ 0 & 1 & 0 \\ 0 & 0 & e^{t} \end{bmatrix} U^{-1},
   \]
where each matrix has an $(n-1)\times (n-1)$ matrix in the middle entry and
   \[
   U=U^*=U^{-1} = \begin{bmatrix} -1/\sqrt2 & 0 &  1/\sqrt2 \\ 0 & 1 & 0 \\ 1/\sqrt2 & 0 &  1/\sqrt2 \end{bmatrix}.
   \]

We have $KA^+K$ decomposition
   \[
   G=KA^+K
   \]
   where $A^+$ consists of $a_t$ for $t\geq0$. Let $M$ be the centralizer of $A$ in $K$. Consider the restricted root space decomposition of $\mathfrak{g}$ with respect $\mathfrak{a}$: 
   \[
   \mathfrak{g} = \mathfrak{a}\oplus \mathfrak{m} \oplus \mathfrak{n} \oplus \overline{\mathfrak{n}}.
   \]
We have
   \[
   \mathfrak{n} =    \mathfrak{n}_1 \oplus    \mathfrak{n}_2, 
   \]
   \[
\mathfrak{n}_1=\{\,U\begin{bmatrix} 0 & 0  & 0 \\ -X/\sqrt2 & 0 & 0 \\ 0 & X^\ast/\sqrt2 & 0  \end{bmatrix}U^{-1} = \begin{bmatrix} 0 & X^\ast/2  & 0 \\  X/2 & 0 & -X/2 \\ 0 & X^\ast/2 & 0  \end{bmatrix} \in    \mathfrak{g}    \mid \, X \in \mathbb{F}^{n-1}  \},
\]
\[
\mathfrak{n}_2=  \{\, U\begin{bmatrix} 0 & 0  & 0 \\  0& 0 & 0 \\ -Z & 0 & 0  \end{bmatrix}U^{-1}=  \begin{bmatrix} Z/2 & 0  & -Z/2 \\  0& 0 & 0 \\ Z/2 & 0 & -Z/2  \end{bmatrix} \in    \mathfrak{g}  \mid \, Z \in \mathrm{Im}(\mathbb{F}) \},
   \]
   and
      \[
   \overline{\mathfrak{n}} =    \overline{\mathfrak{n}}_1 \oplus    \overline{\mathfrak{n}}_2, 
   \]
   \begin{equation}\label{eq_n_1}
\overline{\mathfrak{n}}_1=\{\,U\begin{bmatrix} 0 & X^\ast/\sqrt2  & 0 \\  0 & 0 & -X/\sqrt2 \\ 0 & 0 & 0  \end{bmatrix}U^{-1} = \begin{bmatrix} 0 & -X^\ast/2  & 0 \\  -X/2 & 0 & -X/2 \\ 0 & X^\ast/2 & 0  \end{bmatrix} \in    \mathfrak{g}    \mid \, X \in \mathbb{F}^{n-1}  \},
\end{equation}
   \begin{equation}\label{eq_n_2}
\overline{\mathfrak{n}}_2=  \{\, U\begin{bmatrix} 0 & 0  & -Z \\  0& 0 & 0 \\ 0 & 0 & 0  \end{bmatrix}U^{-1} = \begin{bmatrix} Z/2 & 0  & Z/2 \\  0& 0 & 0 \\ -Z/2 & 0 & -Z/2  \end{bmatrix} \in    \mathfrak{g}  \mid \, Z \in \mathrm{Im}(\mathbb{F}) \}.
\end{equation}
With respect to these expressions, we shall use the coordinates $(X, Z)\in \mathbb{F}^{n-1}\oplus \mathrm{Im}(\F)$ to express elements in $\overline{\mathfrak{n}}$. In these coordinates, the Lie bracket on $\overline{\mathfrak{n}}$ is 
\[
[(X_1, Z_1), (X_2, Z_2)]=(0, \mathrm{Im}(X_1^\ast X_2)).
\]
The exponential map $\exp\colon \overline{\mathfrak{n}} \to \overline N\subset G$ (as well as $\exp\colon {\mathfrak{n}} \to N$) is a diffeomorphism. For $(X, Z)$ in $\overline{\mathfrak{n}}$, we have
\[
\mathrm{exp}(X, Z) =  \begin{bmatrix} 1+X^\ast X/8 + Z/2 & -X^\ast/2 & X^\ast X/8 +Z/2 \\  -X/2 & 1 & -X/2 \\ -X^\ast X/8 - Z/2 & X^\ast/2 & 1-X^\ast X/8-Z/2  \end{bmatrix} \in \overline N \subset G.
\]
We shall identify $\overline{\mathfrak{n}}$ as $\F^{n-1}\oplus \mathrm{Im}(\F)$ and equip it with the standard Euclidean metric. The group $\overline{N}$ is unimodular and we fix a Haar measure given by $d\overline{n}=2^{\mathrm{dim}(\overline{\mathfrak{n}}_2)}dXdZ$ on $\overline{N}\cong \overline{\mathfrak{n}} \cong F^{n-1}\oplus \mathrm{Im}(\F)$ where $dX$ and $dZ$ are the elements of the Lebesgue measures on $\F^{n-1}$ and $\mathrm{Im}(\F)$.  

We have $KAN$ decomposition
   \[
   G=KAN,
   \]
  Langlands decomposition
 \[
 P=MAN,
 \]
 and Bruhat decomposition
  \[
G=Pw P \sqcup P
 \]
 where $w\in K$ is any representative of the non-trivial element in the Weyl group $N_K(A)/Z_K(A)$.

Let $o=[1, 0, \cdots, 0, 1]^T$ in $G/P\subset \mathbb{P}(\F^{n+1})$. Recall that the minimal parabolic subgroup $P$ is the isotropy subgroup of $G$ at $o$. The Cayley transform
\begin{equation}\label{eq_cay}
\mathcal{C}\colon \overline{N} \to G/P
\end{equation}
is defined by sending $n$ in $\overline{N}$ to $n\cdot o=nP$ in $G/P$. The Cayley transform is a diffeomorphism from $\overline{N}$ onto an open dense subset $G/P - \{-o\}$ where $-o=[1, 0, \cdots, 0, -1]^T=w\cdot o$.

Recall that $G/P$ is naturally viewed as the sphere $S^{4n-1}\subset \mathbb{F}^n$ and the point $o$ corresponds to the point $(0, \cdots, 0, 1)$ in the sphere $S^{4n-1}$.

 With this identification of $G/P\subset \mathbb{P}(\F^{n+1})$ and $S^{4n-1}\subset \F^n$, we have for $(X, Z) \in \mathbb{F}^{n-1}\oplus \mathrm{Im}(\F) =\overline{\mathfrak{n}}$, 
\begin{align*}
\mathcal{C}(\exp(X, Z) &)= \exp(X, Z) \cdot o = \begin{pmatrix} -X \\ 1-X^*X/4 - Z \end{pmatrix}  \left( 1+X^*X/4 + Z \right)^{-1}  \\
 & =  \begin{pmatrix} -X(1+X^*X/4 - Z ) \\ 1-(X^*X)^2/16  -|Z|^2 -2Z \end{pmatrix}  \left( (1+X^*X/4)^2 + |Z|^2 \right)^{-1} 
\end{align*}
in $S^{4n-1}\subset \F^n$. Here, the first row represents a vector in $\F^{n-1}$ and the second row represents a vector in $\F$. Setting
\[
||X||^2=X^\ast X, \,\,\, B(X, Z)=(1+||X||^2/4)^2 + |Z|^2,
\]
we have
\begin{equation}\label{eq_cay_formula}
\mathcal{C}(\exp(X, Z)) =  B(X, Z)^{-1} \begin{pmatrix} -X(1+||X||^2/4 - Z ) \\ 1-||X||^4/16  -|Z|^2 -2Z \end{pmatrix}.
\end{equation}
This is essentially the same formula as the one in \cite{ACD04}. In \cite{ACD04}, it is defined from $N$ to $G/P - \{o\}$.

  The Riemannian symmetric space $G/K\subset \mathbb{P}(\F^{n+1})$ has a canonical $G$-invariant metric for which the distance between $eK=[1, 0, \cdots, 0]^T$ and $a_tK=[\cosh t, 0, \cdots, 0, \sinh t]^T=[1, 0, \cdots, 0, \tanh t]^T$ is $|t|$ for $t\in \R$.

For $x, y$ in $G/K$, the Busemann cocycle $\gamma_{x,y}$ is a smooth function on $G/P$ defined as
\[
\gamma_{x,y}(z)=\lim_{z'\to z} \left( d_{G/K}(z', y) -d_{G/K}(z', x) \right)=  \log   \left|  \frac{q(y, z) q(x, x)^{1/2}}{q(x, z)q(y, y)^{1/2}} \right| 
\]
for $z \in G/P$ (see  \cite[Proposition 3.1.1]{CCJJV01}).

Let $d\mu_x$ be the normalized $K_x$ invariant measure on $G/P$ for the isotropy group $K_x$ at $x$ of the $G$-action on $G/K$. We have \cite[Proposition 3.1.2]{CCJJV01}
\[
d\mu_y= e^{-Q\gamma_{x,y}}d\mu_x
\]\
for $Q=\mathrm{dim}(\mathfrak{n}_1)+ 2\mathrm{dim}(\mathfrak{n}_2) = 4(n-1)+2(3)=4n+2$.

We simply write $0$ for the fixed origin $eK$ of $G/K$. Let
\begin{equation}\label{eq_lambda}
\lambda_g=e^{-\gamma_{0, g0}}
\end{equation}
so that
\[
d\mu_{g0}=\lambda_g^Qd\mu_0.
\]
We have for $a_t\in A$ and for $z=[1, z_1, \cdots, z_n]^T$ in $G/P$,
\begin{align*}
\gamma_{0,a_t0}(z) &= \log   \left| 1-(\tanh t)z_n \right|  - (1/2)\log(1-(\tanh t)^2) \\
 & =  \log   \left| \cosh t-(\sinh t)z_n \right| 
\end{align*}
and
\begin{equation}\label{eq_a_t}
\lambda_{a_t}=e^{-\gamma_{0, a_t0}} =    \left| 1-(\tanh t)z_n \right|^{-1} \left( \cosh t \right)^{-1} =   \left| \cosh t-(\sinh t)z_n \right|^{-1}.
\end{equation}

We now briefly recall from \cite{Julg19}, the definition of a $G$-equivariant sub-bundle $E$ of the tangent-bundle $T(G/P)$ and its associated sub-Laplacian $\Delta_E$. Denote by $P_x$, the isotropy subgroup of $G$ at $x$ in $G/P$. The fiber of the cotangent bundle $T^\ast(G/P)$ at $x$ is naturally identified as $(\mathfrak{g}/\mathfrak{p}_x)^\ast$. The latter space is naturally identified as the annihilator $\mathfrak{p}_x^\perp$ of $\mathfrak{p}_x$ in $\mathfrak{g}$ with respect to the Killing form on $\mathfrak{g}$. The space $\mathfrak{p}_x^\perp$ coincides with the nilpotent radical $\mathfrak{n}_x$ of $\mathfrak{p}_x$ which is a 2-step nilpotent Lie algebra with center $\mathfrak{z}_x=[\mathfrak{n}_x, \mathfrak{n}_x]$. A $G$-equivariant sub-bundle $F$ of $T^\ast(G/P)$ is defined so that the fiber at $x$ is $\mathfrak{z}_x$. A $G$-equivariant sub-bundle $E$ of the tangent-bundle $T(G/P)$ is defined as the annihilator $F^\perp$ of $F$. 

Let $d_E$ be the differential operator from $C^\infty(G/P)$ to the section $\Gamma(E^\ast)$ of $E^\ast$, defined as the composition of the de-Rham differential operator and the restriction map from $\Gamma(T^\ast(G/P))$ to $\Gamma(E^\ast)$. The adjoint $d_E^\ast$ is defined with respect to the standard $K$-invariant metric on $G/P$. The $K$-invariant sub-Laplacian on $C^\infty(G/P)$ is defined by $\Delta_E=d_E^\ast d_E$. The sub-Laplacian $\Delta_E$ is a self-adjoint, positive operator on the Hilbert space $L^2(G/P, d\mu_0)$. It has compact resolvent and one-dimensional kernel consisting of constant functions on $G/P$.

For any $\alpha$ in $\R$, let $\hill^{\alpha}(G/P)$ be the Sobolev space defined by the completion of $C^\infty(G/P)$ by the norm
\[
||\xi||_{\hill^{\alpha}(G/P)}= ||(1+\Delta_E)^{\alpha/2}\xi||_{L^2(G/P, d\mu_0)}
\]
for $\xi$ in $C^\infty(G/P)$.
  

 \section{Main result}
As in \cite[Section 8.1]{Julg19},  for $0\leq \mathrm{Re}(s) \leq 1$, let
\[
\pi_s(g)=(1+\Delta_E)^{(Q/4)s}\lambda_g^{(Q/2)(1-s)}\rho(g)(1+\Delta_E)^{-(Q/4)s}
\]
defined on the Hilbert space $\hill^0(G/P)=L^2(G/P, d\mu_0)$. Here, $\rho(g)$ for $g$ in $G$ denotes the left-translation action on the functions on $G/P$ and the cocycle $\lambda_g$ is as in \eqref{eq_lambda}.

Note, 
\[
\pi_0(g)=\lambda_g^{Q/2}\rho(g)
\]
is unitary on $L^2(G/P, d\mu_0)$. This is because 
\[
\rho(g)d\mu_0=(g^{-1})^\ast d\mu_0 = d\mu_{g0}=\lambda_g^Qd\mu_0
\]
so we have
\[
\s{f_1, f_2}_{\hill^0}= \int_{G/P}\overline{f_1}f_2 d\mu_0 =  \int_{G/P}\rho(g)(\overline{f_1}f_2 d\mu_0) 
\]
\[
= \int_{G/P}\rho(g)(\overline{f_1}f_2) \lambda_g^Qd\mu_0 = \s{\pi_0(g)f_1, \pi_0(g)f_2}.
\]
Similarly, if $\mathrm{Re}(s)=0$, $\pi_s(g)$ is unitary on $L^2(G/P, d\mu_0)$.

The following is our main result.

\begin{theorem}\label{thm_main} For any $\epsilon>0$, there is $C=C(\epsilon)>0$ such that for all $s$ satisfying $\mathrm{Re}(s) = 1$, we have the following upper-bound for the operator norm $||\pi_s(g)||$ of $\pi_s(g)$ on $\hill^0(G/P)$ for all $g$ in $G$:
\[
||\pi_s(g)|| \leq C e^{\epsilon l(g)} (1+|\mathrm{Im}(s)|)^{Q/2}.
\]
Here, $l(g)$ is a K-bi-invariant length function on $G$ defined as $l(g)=d_{G/K}(gK, K)$, i.e. $l(ka_tk')=|t|$ for any $k, k'$ in $K$ and $t\in \R$.
\end{theorem}

\begin{corollary}\label{cor_main} For any $\epsilon>0$, there is $C=C(\epsilon)>0$ such that for any $t\in [0, 1]$, we have for all $g$ in $G$,
\[
||\pi_t(g)|| \leq C e^{\epsilon l(g)}.
\]
\end{corollary}
\begin{proof} For any $g$ in $G$ and for any $K$-finite functions $u$ and $v$ in $L^2(G/P)$, we consider a holomorphic function 
\[
s\mapsto \s{\pi_s(g)u, v}
\] 
for $s$ in the strip $0\leq \mathrm{Re}(s) \leq 1$. For each fixed $g$ and $u, v$, it is not hard to see that this holomorphic function is bounded on the strip. Now fix any positive constant $A>0$ and consider a holomorphic function 
\[
s\mapsto \s{\pi_s(g)u, v}e^{As^2}
\] 
on the strip. For each $g$ and $u, v$ fixed, by the three lines theorem, we have
\[
\sup_{\mathrm{Re}(s)=t} |\s{\pi_s(g)u, v}e^{As^2}| \leq  \left( \sup_{\mathrm{Re}(s)=0} |\s{\pi_s(g)u, v}e^{As^2}|  \right)^{1-t} \left( \sup_{\mathrm{Re}(s)=1} |\s{\pi_s(g)u, v}e^{As^2}|  \right)^t
\]
for any $0\leq t\leq 1$. We have
\[
|\s{\pi_t(g)u, v}|e^{At^2} \leq \sup_{\mathrm{Re}(s)=t} |\s{\pi_s(g)u, v}e^{As^2}|.
\]
For $\mathrm{Re}(s)=0$, $\pi_s(g)$ is unitary so we have
\[
\sup_{\mathrm{Re}(s)=0} |\s{\pi_s(g)u, v}e^{As^2}|   \leq  ||u|| ||v||.
\]
For $\mathrm{Re}(s)=1$, by Theorem \ref{thm_main}, 
\begin{align*}
\sup_{\mathrm{Re}(s)=1} |\s{\pi_s(g)u, v}e^{As^2}|  & \leq   C e^{\epsilon l(g)} ||u|| ||v|| \sup_{b \in \R}\left|(1+|b|)^{Q/2}e^{A(1+ib)^2} \right| \\
 & \leq  C' e^{\epsilon l(g)} ||u|| ||v|| 
 \end{align*}
where 
\[
C'=   C  \sup_{b \in \R}\left|(1+|b|)^{Q/2}e^{A(1+ib)^2} \right| < +\infty.
\]
Combining all these, we get 
\[
\sup_{0\leq t\leq 1}|\s{\pi_t(g)u, v}| \leq C' e^{\epsilon l(g)} ||u|| ||v||.
\]
The constant $C'$ does not depend on $g$, $u$ and $v$ so we are done.
\end{proof}


\section{Some reduction}
We begin our proof of Theorem \ref{thm_main}. We will eventually reduce this problem to a certain technical estimate of functions.  First of all,
\begin{proposition}\label{prop_at1} For $s=1$, we have
\[
||\pi_1(g)|| \leq C(1+l(g)) 
\]
for $C>0$ independent of $g$ in $G$, i.e. $\pi_1(g)$ is of linear growth.
\end{proposition}
\begin{proof} Recall \[
\pi_1(g)=(1+\Delta_E)^{(Q/4)}\rho(g)(1+\Delta_E)^{-(Q/4)}.
\]
The left-translation $\rho(g)$ defines a bounded representation of $G$ on the Sobolev space $\hill^{\alpha}(G/P)$. Our assertion is that
\[
||\rho(g)||_{\hill^{Q/2}(G/P)\to \hill^{Q/2}(G/P)} \leq C(1+l(g)). 
\]
This follows from the following two lemmas.
\end{proof}

\begin{lemma} Define a new Hilbert space norm $||\,\,||'$ on $C^\infty(G/P)$ by
\[
||\xi||'^2=||\xi||^2_{L^2(G/P)} + ||(1+\Delta_E)^{Q/4-1/2}d_E\xi||^2_{\Gamma_{L^2}(E^\ast)}.
\]
The norm $||\,\,||'$ and the Sobolev norm $||\,\,||_{\hill^{Q/2}(G/P)}$ are equivalent. Here, $\Gamma_{L^2}(E^\ast)$ is the Hilbert space completion of $\Gamma(E^\ast)$ with respect to the $K$-invariant metric and a K-invariant sub-Laplacian $\Delta_E$ on $\Gamma_{L^2}(E^\ast)$ is defined as $\Delta_E=\nabla_E^\ast\nabla_E$ where a differential operator $\nabla_E\colon \Gamma(E^\ast) \to \Gamma(E^\ast \otimes E^\ast)$ is the composition of a fixed K-invariant connection $\nabla$ on $E^\ast$ and the restriction map from $T^\ast (G/P)$ to $E^\ast$.
\end{lemma}
\begin{proof} Both operators 
\[
(1+\Delta_E)^{-Q/4}(1+d_E^\ast(1+\Delta_E)^{Q/2-1}d_E)^{1/2}, \,\,\, (1+d_E^\ast(1+\Delta_E)^{Q/2-1}d_E)^{-1/2}(1+\Delta_E)^{Q/4}
\]
on $C^\infty(G/P)$ have weighted order zero in a sense of \cite[Section 4.1]{Julg19}, and thus are bounded on $L^2(G/P, d\mu_0)$. These two operators give an equivalence between $\hill^{Q/2}(G/P)$ and the completion of $C^\infty(G/P)$ with respect to $||\,\,||'$. 
\end{proof}

\begin{lemma} The left-translation $\rho(g)$ defines a bounded representation on the completion of $C^\infty(G/P)$ with respect to the new norm $||\,\,||'$ and it has linear growth.
\end{lemma}
\begin{proof} Denote by $W'$, the completion of $C^\infty(G/P)$ with respect to $||\,\,||'$. Note that $W'$ contains the trivial sub-representation $\C 1_{G/P}$ spanned by the constant functions. Our claim follows from the fact that the induced representation on the quotient space $W_0=W'/\C1_{G/P}$ is uniformly bounded. This fact was proved in \cite[Corollary 4.6]{N20} by observing that $W_0$ is naturally viewed as a sub-representation of the representation of $G$ on (the completion of) $\Gamma(E^\ast)$ with respect to the norm $||(1+\Delta_E)^{Q/4-1/2}\xi||$ for $\xi \in \Gamma(E^\ast)$. This representation of $G$ on $\Gamma(E^\ast)$ is a principal series representation and it is uniformly bounded with respect to the given norm (See \cite[Theorem 4.5]{N20}).
\end{proof}

Let $s=1+ib$ for $b\in \R$. We have 
\[
\pi_s(g)=(1+\Delta_E)^{(Q/4)ib}(1+\Delta_E)^{(Q/4)}\lambda_g^{-(Q/2)(ib)}\rho(g)(1+\Delta_E)^{-(Q/4)}(1+\Delta_E)^{-(Q/4)ib},
\]
so
\begin{align*}
||\pi_s(g)||& =||(1+\Delta_E)^{(Q/4)}\lambda_g^{-(Q/2)(ib)}\rho(g)(1+\Delta_E)^{-(Q/4)}|| \\
& = ||(1+\Delta_E)^{(Q/4)}\lambda_g^{-(Q/2)(ib)}(1+\Delta_E)^{-(Q/4)} \pi_1(g)||.
\end{align*}

In view of Proposition \ref{prop_at1}, using $G=KA^+K$, in order to prove Theorem \ref{thm_main}, we just need to show, for any $\epsilon>0$, the existence of $C=C(\epsilon)>0$ such that for $t\geq0$ and for $b\in \R$,
\begin{equation*}
||(1+\Delta_E)^{(Q/4)}\lambda_{a_t}^{-(Q/2)(ib)} (1+\Delta_E)^{-(Q/4)}|| \leq   Ce^{\epsilon t} (1+|b|)^{Q/2}.
\end{equation*}

Let $\chi$ be a smooth function on $G/P$ which is $1$ near the point $o=[1, 0, \cdots, 0, 1]^T$ and vanishes outside a small neighborhood of $o$. 
\begin{lemma}  There is $C>0$ such that for $t\geq0$ and for $b\in \R$,
\[
||(1+\Delta_E)^{(Q/4)}(1-\chi)\lambda_{a_t}^{-(Q/2)(ib)}(1+\Delta_E)^{-(Q/4)}|| \leq C (1+|b|)^{Q/2}.
\]
\end{lemma}
\begin{proof}  Recall from \eqref{eq_a_t}
\[
\lambda_{a_t} =   \left| 1-(\tanh t)z_n \right|^{-1} \left(\cosh t \right)^{-1}
\]
so
\[
\lambda_{a_t}^{-(Q/2)(ib)}=     \left| 1-(\tanh t)z_n \right|^{(Q/2)ib}  \left(\cosh t \right)^{Q/2(ib)}.
\]
Note that the second factor is a constant function on $G/P$ of modulus one and that
\[
(1-\chi) \left| 1-(\tanh t)z_n \right|^{(Q/2)ib}
\]
smoothly converges as $t\to +\infty$ to the smooth function 
\[
(1-\chi) \left| 1-z_n \right|^{(Q/2)ib}
\]
on $G/P$. The claim follows from these.
\end{proof}

Thus, we reduced our proof of Theorem \ref{thm_main} to the following.
\begin{lemma}\label{lem_chi_bound} Let $\chi$ be a smooth function on $G/P$ which is supported near the point $o$ ($z_n=1$). Consider the multiplication operator
\[
\chi \lambda_{a_t}^{-(Q/2)(ib)}\colon \hill^{Q/2}(G/P) \to \hill^{Q/2}(G/P).
\]
For any $\epsilon>0$, there is $C=C(\epsilon)>0$ such that for $t\geq0$ and for $b\in \R$,
\[
||\chi \lambda_{a_t}^{-(Q/2)(ib)}||_{\hill^{Q/2}(G/P) \to \hill^{Q/2}(G/P)} \leq C e^{\epsilon t} (1+|b|)^{Q/2}.
\]
\end{lemma}

We will prove Lemma \ref{lem_chi_bound} using a local model $\overline{N}\cong N$ of $G/P$ around the point $o$ via the Cayley transform $\mathcal{C}$ \eqref{eq_cay}.

Fix an orthonormal basis $\{E_j\}_{1\leq j \leq \mathrm{dim}(\overline{\mathfrak{n}}_1)}$ of $\overline{\mathfrak{n}}_1$. We identify $\overline{\mathfrak{n}}$ as the left-invariant vector field on $\overline{N}$. The sub-Laplacian $\Delta_{\overline{\mathfrak{n}}}$ on $\overline{N}$ is defined by
\[
\Delta_{\overline{\mathfrak{n}}}= -\sum_{j=1}^{ \mathrm{dim}(\overline{\mathfrak{n}}_1)}E_j^2
\]
on the space $C_c^\infty(\overline{N})$ of compactly supported smooth functions on $\overline{N}$. It defines  an essentially self-adjoint, positive operator on $L^2(\overline{N})$. For any $\alpha\in \R$, we define the homogeneous Sobolev space $\dot\hill^\alpha(\overline{N})$ as the completion of $C_c^\infty(\overline{N})$ by the norm
\[
||\xi||_{\dot\hill^\alpha(\overline{N})} = ||\Delta_{\overline{\mathfrak{n}}}^{\alpha/2}\xi||_{L^2(\overline{N})}.
\]
For a technical purpose, we also define the non-homogenous Sobolev space $\hill^\alpha(\overline{N})$ as the completion of $C_c^\infty(\overline{N})$ by the norm
\[
||\xi||_{\hill^\alpha(\overline{N})} = ||(1+\Delta_{\overline{\mathfrak{n}}})^{\alpha/2}\xi||_{L^2(\overline{N})}.
\]
Multiplication by any compactly supported smooth function on $\overline{N}$ defines a bounded operator on $\hill^\alpha(\overline{N})$ for any $\alpha$ \cite[Corollary 4.15]{Folland75}. It is also bounded on the homogeneous Sobolev space $\dot\hill^\alpha(\overline{N}$) for $-Q/2<\alpha<Q/2$ but not necessarily for $|\alpha|\geq Q/2$.

For any open subset $U$ of $\overline{N}$, we define $\hill^{\alpha}(U)$ and $\dot\hill^{\alpha}(U)$ as the closures of $C^\infty_c(U)$ in $\hill^{\alpha}(\overline{N})$ and  $\dot\hill^{\alpha}(\overline{N})$ respectively. Similarly, for any open subset $U_{G/P}$ of $G/P$, we define $\hill^\alpha(U_{G/P})$ as the closure of $C^\infty_c(U_{G/P})$ in $\hill^{\alpha}(G/P)$.

 We note two things. First, if $U\subset \overline{N}$ is relatively compact, the inclusion $\hill^{\alpha}(U)$ to $\dot \hill^{\alpha}(U)$ is an isomorphism for all $\alpha \geq 0$. This is because the sub-Laplacian $\Delta_{\overline{\mathfrak{n}}}$ is locally, bounded away from zero. Secondly, for any $\alpha$, locally but not globally, the Cayley transform $\mathcal{C}\colon \overline{N} \to G/P$ induces an isomorphism on Sobolev spaces $\hill^{\alpha}$. That is, the composition by the Cayley transform $\mathcal{C}$ induces an isomorphism from $\hill^{\alpha}(U_{G/P})$ to $\hill^{\alpha}(U_{\overline{N}})$ for any relatively compact open subset $U_{\overline{N}}$ and for $U_{G/P}=\mathcal{C}U_{\overline{N}}$. We only use this fact for $\alpha=Q/2$. 

Thus, Lemma \ref{lem_chi_bound} follows from the following claim. Notice that we are using the non-homogenous Sobolev space in this statement.

\begin{claim}\label{cl_localver} Let $\chi$ be a smooth compactly supported function on $\overline{N}$. For any $\epsilon>0$, there is $C=C(\epsilon)>0$ such that for $t\geq0$ and for $b\in \R$, the multiplication operator
\[
\chi  (\lambda_{a_t}\circ \mathcal{C})^{-(Q/2)(ib)}  \colon \hill^{Q/2}(\overline{N}) \to \hill^{Q/2}(\overline{N}) 
\]
has norm bound
\[
||\chi  (\lambda_{a_t}\circ \mathcal{C})^{-(Q/2)(ib)} ||_{\hill^{Q/2}(\overline{N}) \to \hill^{Q/2}(\overline{N})} \leq C e^{\epsilon t} (1+|b|)^{Q/2}.
\]
\end{claim}

In order to prove this claim, we give an explicit expression of $\lambda_{a_t}\circ \mathcal{C}$ on $\overline{N}$. From now on, when there is no confusion, we identify $\overline{N}$ with $\overline{\mathfrak{n}}$ via the exponential map and use the coordinates $(X, Z)\in \F^{n-1}\oplus \mathrm{Im}(\F)$ for $\overline{\mathfrak{n}}$ to express the element $\exp(X, Z)$ in $\overline{N}$ as well.

\begin{lemma}\label{lem_cocycle_formula}(c.f.  \cite[Corollary 3.9]{ACD04}) For any $t \in \R$, for any $(X, Z)\in \F^{n-1}\oplus \mathrm{Im}(\F)$, we have 
\[
(\lambda_{a_t}\circ \mathcal{C})(X, Z) = e^{-t}\frac{B(X, Z)^{1/2}}{B_t(X, Z)^{1/2}}
\]
where
\[
B(X, Z)=(1+||X||^2/4)^2 + |Z|^2,  \,\, B_t(X, Z)=(e^{-2t}+||X||^2/4)^2 + |Z|^2.
\]
\end{lemma}
\begin{proof} Recall that
\[
d\mu_{a_t0}=(a_{-t})^\ast d\mu_0= \lambda_{a_t}^Qd\mu_0.
\]
That is, $ \lambda_{a_t}^Q$ is the Jacobian $J_{a_{-t}}$ of the map $a_{-t}$ on $G/P$ (with respect to $d\mu_0$). Let $\delta_{t}\colon \overline{N}\to \overline{N}$ be the map $(X, Z)\to (e^tX, e^{2t}Z)$. We have
\[
\mathrm{Ad}_{a_{-t}}=\delta_{t}
\]
on $\overline{N}$. By the chain rule, we have
\[
\lambda_{a_t}^Q\circ \mathcal{C} = J_{a_{-t}}\circ\mathcal{C} =J_{\mathcal{C}}\circ \delta_t \cdot J_{\delta_t} \cdot J_{\mathcal{C}}^{-1}
\]
on $\overline{N}$ where $J_{\mathcal{C}}$ and $J_{\delta_t}$ are the Jacobian of the map $\mathcal{C}$ and $\delta_t$ respectively. We have
\[
J_{\delta_t}=e^{Qt}
\]
and
\[
J_{\mathcal{C}}= B(X, Z)^{-Q/2}
\]
(see the end of Section 2.3 of \cite{ACD04}). Thus,
\[
\lambda_{a_t}^Q\circ \mathcal{C}  = J_{\mathcal{C}}\circ \delta_t \cdot J_{\delta_t} \cdot J_{\mathcal{C}}^{-1} =  e^{Qt} \left( \frac{B(X, Z)}{B(e^{t}X, e^{2t}Z)} \right)^{Q/2}.
\]
We get
\[
\lambda_{a_t}\circ \mathcal{C}  = e^{t} \frac{B(X, Z)^{1/2}}{B(e^tX, e^{2t}Z)^{1/2}} =  e^{-t}\frac{B(X, Z)^{1/2}}{B_t(X, Z)^{1/2}}.
\]
\end{proof}

\begin{remark}
Recall
\[
\lambda_{a_t}=e^{-\gamma_{0,a_t0}}=\vert \cosh t-\sinh t\vert^{-1}=\vert 1-(\tanh t) z_n\vert^{-1}(\cosh t)^{-1}.
\]
The link with Lemma \ref{lem_cocycle_formula} is the following. For $z_n =z_n\circ \mathcal{C}(X, Z)$, using \eqref{eq_cay_formula},
\begin{align*}
z_n = &B^{-1}(1-||X||^4/16  -|Z|^2 -2Z ) =B^{-1}(-N^2+2N -|Z|^2 -2Z ) \\
 & = -1+2B^{-1}N-2B^{-1}Z
\end{align*}
where $N=1+\Vert X\Vert^2/4$ and $B=N^2+\vert Z\vert^2$. An easy calculation yields (denote $r=\tanh t$)
\[
\lambda_{a_t}^{-2}\circ \mathcal{C}=(1-r^2)^{-1}\vert 1-rz_n\vert^2={1+r\over 1-r}(1-{4r\over 1+r}B^{-1}N+{4r^2\over (1+r)^2}B^{-1}).
\]
On the other hand 
\[
B_t= (\Vert X\Vert^2/4+e^{-2t})^2+\vert Z\vert^2=
(N-{2r\over 1+r})^2+\vert Z\vert^2=B-{4r\over 1+r}N+{4r^2\over (1+r)^2},
\]
so that $\lambda_{a_t}^{-2}\circ\mathcal{C} =e^{2t}B^{-1}B_t$ which coincides with the above formula.
\end{remark}

In the light of the formula in Lemma \ref{lem_cocycle_formula}, we have
\[
  (\lambda_{a_t}\circ \mathcal{C})^{-(Q/2)(ib)}  =  e^{Q/2(t)(ib)} B(X, Z)^{-Q/4(ib)} B_t(X, Z)^{(Q/4)ib}.
  \]
Note that the first factor $e^{Q/2(t)(ib)}$ is a constant function on $\overline{N}$ of modulus one. We reduced Claim \ref{cl_localver}, and hence Theorem \ref{thm_main}, to the following.

\begin{lemma}\label{lem_localver} Let $\chi$ be a smooth compactly supported function on $\overline{N}$. For any $\epsilon>0$, there is $C=C(\epsilon)>0$ such that for $t\geq0$ and for $b\in \R$, the multiplication operator
\[
\chi  B_t(X, Z)^{ib}B(X, Z)^{-ib} \colon \hill^{Q/2}(\overline{N}) \to \hill^{Q/2}(\overline{N}) 
\]
has norm bound
\[
||\chi  B_t(X, Z)^{ib}B(X, Z)^{-ib} ||_{\hill^{Q/2}(\overline{N}) \to \hill^{Q/2}(\overline{N})} \leq C e^{\epsilon t} (1+|b|)^{Q/2}.
\]
\end{lemma}

From now on, we focus on proving the lemma \ref{lem_localver}. We have the following easy estimates for $B_t(X, Z)$:
\begin{enumerate}
\item $B_t(X, Z)\geq e^{-4t}$,
\item $B_t(X, Z)\geq ||X||^4/16 + |Z|^2 = \cN^4$,
\end{enumerate}
where 
\[
\cN= (||X||^4/16 + |Z|^2)^{1/4}
\]
is a homogeneous gauge on $\overline{N}$.

\section{Proof of Lemma \ref{lem_localver}}
 
 Let us follow some convention from \cite{ACD04}: for a positive integer $k\in \Z$, $\mathfrak{D}^k$ denotes the set of all differential operators of the form
 \[
 E_{j_1}E_{j_2}\cdots E_{j_k},
 \]
 where $1\leq j_1, j_2, \cdots, j_k\leq \mathrm{dim}(\overline{\mathfrak{n}}_1)$. The following result from \cite{ACD04} will be the key. Note that the homogeneous Sobolev space is used in this statement.
 
 \begin{lemma}\cite[Theorem 3.6]{ACD04}\label{lem_homog} Let $0\leq \alpha \leq \beta<Q/2$. Let $m$ be a smooth function on $\overline{N}-\{0\}$. Suppose that on $\overline{N}-\{0\}$,
 \[
 |D^j m| \leq C_j\cN^{-d-j}\,\,\,\, 
 \] 
 holds for any $0\leq j \leq k= \left \lceil{\alpha}\right \rceil $, $D^j\in \mathfrak{D}^j$ and $d=\beta-\alpha$. Then, 
 \[
 ||m||_{\dot\hill^\beta(\oN) \to \dot\hill^\alpha(\oN)} \leq C(\alpha, \beta)(C_0+C_1+\cdots+C_k)
 \] 
 where the constant $C(\alpha, \beta)$ only depends on $\alpha$ and $\beta$.
\end{lemma}
\begin{proof} This is what is proven in \cite[Theorem 3.6]{ACD04} and used in the proof of \cite[Corollary 3.9]{ACD04}. For the detail of its proof, see the proof of \cite[Theorem 7]{AD05}.
\end{proof}

The following technical estimate for the derivatives of $B_t$ very roughly means $B_t$ behaves like a homogeneous function of degree $4$.

\begin{lemma}\label{lem_tech} Let $U$ be a relatively compact open subset of $\oN$. For any positive integer $s\geq0$, there is $C=C(s)>0$ so that for any $D^s\in \mathfrak{D}^s$ and for any $(X, Z)$ in $U$,
\[
|D^s(B_t^{ib})|(X, Z)\leq C \frac{1}{B_t(X, Z)^{s/4}}(1+|b|)^{s}
\]
for all $t\geq0$ and $b\in \R$. 
\end{lemma}
\begin{proof} The proof will be given in the next (last) section.
\end{proof}
Now we use Lemma \ref{lem_homog} together with Lemma \ref{lem_tech} to prove Lemma \ref{lem_localver}, that is we show that for any smooth compactly supported function $\chi$ on $\oN$ and for any $\epsilon>0$, the multiplication operator
\[
\chi B_t^{ib}B^{-ib} \colon \hill^{Q/2}(\oN) \to \hill^{Q/2}(\oN) 
\]
has norm bound
\[
|| \chi B_t^{ib}B^{-ib} ||_{\hill^{Q/2}(\oN) \to \hill^{Q/2}(\oN)} \leq C e^{\epsilon t} (1+|b|)^{Q/2},
\]
for $C=C(\epsilon)>0$ independent of $t\geq0$ and $b\in\R$.

Note that for $\xi$ in $\hill^{Q/2}(\oN)$, the following two norms
\[
||\xi||_{\hill^{Q/2}(\oN)}, \,\,\,  ||\xi||_{\hill^0(\oN)} + \sum_{j=1}^{\mathrm{dim}(\overline{\mathfrak{n}}_1)} ||E_j \xi||_{ \hill^{Q/2-1}(\oN)}
\]
are equivalent. Using this, we see that it is enough to show that for any compactly supported smooth function $\chi$ and for any $\epsilon>0$, there is $C=C(\epsilon)>0$ so that for any $t\geq0$ and $b\in \R$,
\begin{equation}\label{eq_estimate1}
|| \chi B_t^{ib}B^{-ib}||_{\hill^{Q/2-1}(\oN) \to \hill^{Q/2-1}(\oN)} \leq C e^{\epsilon t} (1+|b|)^{Q/2},
\end{equation}
and 
\begin{equation}\label{eq_estimate2}
||  \chi E_j(B_t^{ib}B^{-ib}) ||_{\hill^{Q/2}(\oN) \to \hill^{Q/2-1}(\oN)} \leq C e^{\epsilon t} (1+|b|)^{Q/2}.
\end{equation}

By Lemma \ref{lem_tech}, on a ball $U$ which contains the support of $\chi$,
\[
|D^s(B_t^{ib})|\leq C_s B_t^{-s/4}(1+|b|)^{s} \leq  C_s(1+|b|)^{s}\cN^{-s}.
\]
It follows that there is $C>0$ such that for $0\leq s \leq Q/2-1$,
\[ 
|D^s( \chi B_t^{ib}B^{-ib})| \leq C(1+|b|)^{s}\cN^{-s}.
\]

By Lemma \ref{lem_homog}, it follows that for some $C>0$,
\[
||\chi B_t^{ib}B^{-ib}||_{\dot\hill^{Q/2-1}(\oN) \to \dot\hill^{Q/2-1}(\oN)} \leq C(1+|b|)^{Q/2-1} \leq C e^{\epsilon t} (1+|b|)^{Q/2}
\]
for all $t\geq 0$ and $b\in\R$. Since $\chi$ is compactly supported, we can replace the homogeneous Sobolev space $\dot\hill^\alpha$ by the non-homogeneous Sobolev space $\hill^\alpha$ in this inequality. This proves \eqref{eq_estimate1}.

As for  \eqref{eq_estimate2}, first note that for any compactly supported smooth function $\chi_0$,
 \[
 \chi_0\colon \hill^{Q/2}(\oN) \to \hill^{Q/2-\epsilon}(\oN)
 \]
 is bounded for any $\epsilon>0$. Taking $\chi_0$ such that $\chi_0\chi=\chi$, it is enough to bound the norm
\[
 \chi E_j(B_t^{ib}B^{-ib})  \colon  \dot\hill^{Q/2-\epsilon}(\oN) \to \dot\hill^{Q/2-1}(\oN).
\]
Again, we are using that $\chi$ is compactly supported to replace the homogeneous Sobolev space $\dot\hill^\alpha$ by the non-homogeneous Sobolev space $\hill^\alpha$.

Let $\beta=Q/2-\epsilon$, $\alpha=Q/2-1$ and $d=\beta-\alpha=1-\epsilon$. By Lemma \ref{lem_tech}, we have on a ball $U$ which contains the support of $\chi$,
\begin{align*}
 & |D^s(E_j(B_t^{ib}))|  = |D^{s+1}B_t^{ib}| \leq  C_{s+1}B_t^{-(s+1)/4} (1+|b|)^{s+1} \\
 & \leq   C_{s+1}B_t^{-\epsilon/4}  B_t^{-(s+1-\epsilon)/4}(1+|b|)^{s+1}     \leq C_{s+1}B_t^{-\epsilon/4} \cN^{-d-s}(1+|b|)^{s+1}.
\end{align*}
Moreover, we have
\[
B_t^{-\epsilon/4} \leq e^{\epsilon t}
\]
for $t\geq0$. It follows that there is $C>0$ such that for $0\leq s \leq Q/2-1$,
\[ 
|D^s( \chi E_j(B_t^{ib}B^{-ib}))| \leq Ce^{\epsilon t} \cN^{-d-s}(1+|b|)^{s+1}.
\]
Hence, by Lemma \ref{lem_homog}, it follows that for some $C>0$,
\[
||\chi E_j(B_t^{ib}B^{-ib}) ||_{\dot\hill^{Q/2-\epsilon}(\oN) \to \dot\hill^{Q/2-1}(\oN)} \leq C(1+|b|)^{Q/2}e^{\epsilon t}.
\]
This gives \eqref{eq_estimate2}.  This proves Lemma \ref{lem_localver}, and hence Theorem \ref{thm_main}, modulo the technical estimate, Lemma \ref{lem_tech}.



\section{Technical Estimate}

We give a proof of Lemma \ref{lem_tech}. We fix a relatively compact open set $U$ of $\oN$ and let $C_0>0$ so that
\[
B_t(X, Z) = (e^{-2t}+||X||^2/4)^2 + |Z|^2 \leq C_0 
\]
for all $t\geq0$.

We need a small lemma. Again, this roughly means $B_t$ behaves like a homogeneous function of degree $4$.

\begin{lemma}\label{lem_prep} On $\oN$, we have 
\begin{enumerate}
\item $|D^1B_t|\leq C_1B_t^{3/4}$,
\item $|D^2B_t|\leq C_2B_t^{2/4}$,
\item $|D^3B_t|\leq C_3B_t^{1/4}$,
\item $|D^4B_t|\leq C_4$
\item $D^5B_t=0 \,\,\, (j\geq5)$,
\end{enumerate}
for some constants $C_j\geq0$ independent of $t\geq0$ and of $D^j\in \mathfrak{D}^j$.
\end{lemma}
\begin{proof}  
Let us write
\[
B_t = e^{-4t} + e^{-2t}||X||^2/2 + P_4
\]
where $P_4=\cN^{4}=||X||^4/16 + |Z|^2$.
We have 
\begin{align*}
D^1B_t &= e^{-2t}D^1(||X||^2)/2 + P_3 \\
D^2B_t & =  e^{-2t}D^2(||x||^2)/2   + P_2 \\
D^3B_t & =  P_1 \\
D^4 B_t &=  P_0
\end{align*}
where $P_d$ is a homogenous polynomial of degree $d$. Note that for each such $P_d$, we have
\[
|P_d| \leq C'_d\cN^{d} \leq C'_d B_t^{d/4}
\]
for some $C'_d>0$.  We also have 
\begin{align*}
|e^{-2t}D^1(||X||^2)/2| & \leq 2B_t^{3/4} \\
| e^{-2t}D^2(||X||^2)/2| & \leq B_t^{2/4}.
\end{align*}
The first one follows from 
\[
B_t^3 \geq  (e^{-4t} + ||X||^4/16)^3\geq e^{-8t}||X||^4/16.
\]
The second one follows from
\[
B_t \geq e^{-4t}.
\]
The claim follows from these.
\end{proof} 

\begin{proof}[Proof of Lemma \ref{lem_tech}]
For any $s>0$, we need to show that for some $C=C(s)>0$, on $U$,
\[
|D^s(B_t^{ib})|\leq C \frac{1}{(B_t)^{s/4}}(1+|b|)^{s}.
\]
For $s=0$, this is trivial.  For $s=1$, we have, on $U$,
\[
|D^1(B_t^{ib})|=|D^1e^{ib\log (B_t)}| = \frac{|D^1(B_t)|}{B_t}|b| \leq C_1\frac{ B_t^{3/4}}{B_t} (1+ |b|)= C_1B_t^{-1/4}(1+|b|).
\]
The general case for $s\geq1$ follows from that $D^se^{ib\log (B_t)}$ is a finite sum of the product
 \[
 (ib)^le^{ib\log (B_t)} \prod_{1\leq j\leq l}(D^{k_j}\log(B_t))  
\]
where $k_1+\cdots+k_l=s$ ($k_j\geq1$) and $D^{k_j} \in \mathfrak{D}^{k_j}$ and that on $U$,
\[
|D^j(\log (B_t))| \leq C'_j (B_t)^{-j/4}\,\,\, (j\geq1).
\]
We can see the latter from that $D^s(\log (B_t))$ is a finite sum of the product
\[
(B_t)^{-(m_1+\cdots+m_s)} \prod_{1\leq j \leq s}(D^{j}B_t)^{m_j}
\]
where $m_1+2m_2+\cdots + sm_s=s$ and $D^{j} \in \mathfrak{D}^{j}$ and that 
\[
\left|\frac{D^jB_t}{B_t} \right| \leq C_j(B_t)^{-j/4},
\]
that is for $j\geq1$,
\[
|D^jB_t| \leq C_j(B_t)^{1-j/4}.
\]
The last one follows from Lemma \ref{lem_prep}. In this argument, we used that $B_t\leq C_0$ on $U$ thus 
\[
B_t^{-s_1} \leq C_0^{s_2-s_1} B_t^{-s_2}
\]
on $U$ for $0\leq s_1\leq s_2$.

\end{proof}


\bibliography{Refs}
\bibliographystyle{plain}

\end{document}